\documentclass[12pt]{article}

\usepackage{amsmath}
\usepackage[normalem]{ulem} %
\usepackage{mathtools}
\usepackage{amsfonts}
\usepackage{accents}
\usepackage[mathscr]{eucal}
\usepackage{amsthm}
\usepackage{amssymb}
\usepackage{enumerate}
\usepackage{cases}
\usepackage{multirow}
\usepackage{subcaption}
\usepackage{graphicx}
\usepackage{pgf,tikz}
\usepackage{color}
\usepackage[bookmarks=true, bookmarksopen=true, bookmarksopenlevel=4]{hyperref}
\definecolor{myblue}{rgb}{0,0,0.6}     %
\hypersetup{pdftitle={Notes on Wang Theorem},
            pdfauthor={Hewett},
     colorlinks=true, linkcolor=myblue,  citecolor=myblue, filecolor=myblue,   urlcolor=myblue,  }
\definecolor{dhcol}{rgb}{0,0.5,0}
\providecommand{\dhc}[1]{#1}
\definecolor{sccol}{rgb}{0,0,0.5}

\graphicspath{{figs/}}
\begin{document}
\newcommand{\rf}[1]{(\ref{#1})}
\newcommand{\mmbox}[1]{\fbox{\ensuremath{\displaystyle{ #1 }}}}	%

\newcommand{\hs}[1]{\hspace{#1mm}}
\newcommand{\vs}[1]{\vspace{#1mm}}

\newcommand{\ri}{{\mathrm{i}}}
\newcommand{\re}{{\mathrm{e}}}
\newcommand{\rd}{\mathrm{d}}

\newcommand{\R}{\mathbb{R}}
\newcommand{\Q}{\mathbb{Q}}
\newcommand{\N}{\mathbb{N}}
\newcommand{\Z}{\mathbb{Z}}
\newcommand{\C}{\mathbb{C}}
\newcommand{\K}{{\mathbb{K}}}

\newcommand{\cA}{\mathcal{A}}
\newcommand{\cB}{\mathcal{B}}
\newcommand{\cC}{\mathcal{C}}
\newcommand{\cS}{\mathcal{S}}
\newcommand{\cD}{\mathcal{D}}
\newcommand{\cH}{\mathcal{H}}
\newcommand{\cI}{\mathcal{I}}
\newcommand{\cItilde}{\tilde{\mathcal{I}}}
\newcommand{\cIhat}{\hat{\mathcal{I}}}
\newcommand{\cIcheck}{\check{\mathcal{I}}}
\newcommand{\cIstar}{{\mathcal{I}^*}}
\newcommand{\cJ}{\mathcal{J}}
\newcommand{\cM}{\mathcal{M}}
\newcommand{\cP}{\mathcal{P}}
\newcommand{\cV}{{\mathcal V}}
\newcommand{\cW}{{\mathcal W}}
\newcommand{\scrD}{\mathscr{D}}
\newcommand{\scrS}{\mathscr{S}}
\newcommand{\scrJ}{\mathscr{J}}
\newcommand{\sD}{\mathsf{D}}
\newcommand{\sN}{\mathsf{N}}
\newcommand{\sS}{\mathsf{S}}
 \newcommand{\sT}{\mathsf{T}}
 \newcommand{\sH}{\mathsf{H}}
 \newcommand{\sI}{\mathsf{I}}
 
\newcommand{\bs}[1]{\mathbf{#1}}
\newcommand{\bb}{\mathbf{b}}
\newcommand{\bd}{\mathbf{d}}
\newcommand{\bn}{\mathbf{n}}
\newcommand{\bp}{\mathbf{p}}
\newcommand{\bP}{\mathbf{P}}
\newcommand{\bv}{\mathbf{v}}
\newcommand{\bx}{\mathbf{x}}
\newcommand{\by}{\mathbf{y}}
\newcommand{\bz}{{\mathbf{z}}}
\newcommand{\bxi}{\boldsymbol{\xi}}
\newcommand{\boldeta}{\boldsymbol{\eta}}	%

\newcommand{\ts}{\tilde{s}}
\newcommand{\tGamma}{{\tilde{\Gamma}}}
 \newcommand{\tbx}{\tilde{\bx}}
 \newcommand{\tbd}{\tilde{\bd}}
 \newcommand{\txi}{\xi}
 
\newcommand{\done}[2]{\dfrac{d {#1}}{d {#2}}}
\newcommand{\donet}[2]{\frac{d {#1}}{d {#2}}}
\newcommand{\pdone}[2]{\dfrac{\partial {#1}}{\partial {#2}}}
\newcommand{\pdonet}[2]{\frac{\partial {#1}}{\partial {#2}}}
\newcommand{\pdonetext}[2]{\partial {#1}/\partial {#2}}
\newcommand{\pdtwo}[2]{\dfrac{\partial^2 {#1}}{\partial {#2}^2}}
\newcommand{\pdtwot}[2]{\frac{\partial^2 {#1}}{\partial {#2}^2}}
\newcommand{\pdtwomix}[3]{\dfrac{\partial^2 {#1}}{\partial {#2}\partial {#3}}}
\newcommand{\pdtwomixt}[3]{\frac{\partial^2 {#1}}{\partial {#2}\partial {#3}}}
\newcommand{\bnabla}{\boldsymbol{\nabla}}
\newcommand{\dive}{\boldsymbol{\nabla}\cdot}
\newcommand{\curl}{\boldsymbol{\nabla}\times}
\newcommand{\Phixy}{\Phi(\bx,\by)}
\newcommand{\PhiOxy}{\Phi_0(\bx,\by)}
\newcommand{\dxPhixy}{\pdone{\Phi}{n(\bx)}(\bx,\by)}
\newcommand{\dyPhixy}{\pdone{\Phi}{n(\by)}(\bx,\by)}
\newcommand{\dxPhiOxy}{\pdone{\Phi_0}{n(\bx)}(\bx,\by)}
\newcommand{\dyPhiOxy}{\pdone{\Phi_0}{n(\by)}(\bx,\by)}

\newcommand{\eps}{\varepsilon}
\newcommand{\real}[1]{{\rm Re}\left[#1\right]} %
\newcommand{\imag}[1]{{\rm Im}\left[#1\right]}
\newcommand{\ol}[1]{\overline{#1}}
\newcommand{\ord}[1]{\mathcal{O}\left(#1\right)}
\newcommand{\oord}[1]{o\left(#1\right)}
\newcommand{\Ord}[1]{\Theta\left(#1\right)}

\newcommand{\hsnorm}[1]{||#1||_{H^{s}(\bs{R})}}
\newcommand{\hnorm}[1]{||#1||_{\tilde{H}^{-1/2}((0,1))}}
\newcommand{\norm}[2]{\left\|#1\right\|_{#2}}
\newcommand{\normt}[2]{\|#1\|_{#2}}
\newcommand{\on}[1]{\Vert{#1} \Vert_{1}}
\newcommand{\tn}[1]{\Vert{#1} \Vert_{2}}

\newcommand{\xt}{\mathbf{x},t}
\newcommand{\PhiF}{\Phi_{\rm freq}}
\newcommand{\cone}{{c_{j}^\pm}}
\newcommand{\ctwo}{{c_{2,j}^\pm}}
\newcommand{\cthree}{{c_{3,j}^\pm}}

\newtheorem{thm}{Theorem}[section]
\newtheorem{lem}[thm]{Lemma}
\newtheorem{defn}[thm]{Definition}
\newtheorem{prop}[thm]{Proposition}
\newtheorem{cor}[thm]{Corollary}
\newtheorem{rem}[thm]{Remark}
\newtheorem{conj}[thm]{Conjecture}
\newtheorem{ass}[thm]{Assumption}
\newtheorem{example}[thm]{Example} %

\newcommand{\tH}{\widetilde{H}}
\newcommand{\Hze}{H_{\rm ze}} 	%
\newcommand{\uze}{u_{\rm ze}}		%
\newcommand{\dimH}{{\rm dim_H}}
\newcommand{\dimB}{{\rm dim_B}}
\newcommand{\IntClosOm}{\mathrm{int}(\overline{\Omega})}
\newcommand{\IntClosOmOne}{\mathrm{int}(\overline{\Omega_1})}
\newcommand{\IntClosOmTwo}{\mathrm{int}(\overline{\Omega_2})}
\newcommand{\Ccomp}{C^{\rm comp}}
\newcommand{\tCcomp}{\tilde{C}^{\rm comp}}
\newcommand{\uC}{\underline{C}}
\newcommand{\utC}{\underline{\tilde{C}}}
\newcommand{\oC}{\overline{C}}
\newcommand{\otC}{\overline{\tilde{C}}}
\newcommand{\capcomp}{{\rm cap}^{\rm comp}}
\newcommand{\Capcomp}{{\rm Cap}^{\rm comp}}
\newcommand{\tcapcomp}{\widetilde{{\rm cap}}^{\rm comp}}
\newcommand{\tCapcomp}{\widetilde{{\rm Cap}}^{\rm comp}}
\newcommand{\hcapcomp}{\widehat{{\rm cap}}^{\rm comp}}
\newcommand{\hCapcomp}{\widehat{{\rm Cap}}^{\rm comp}}
\newcommand{\tcap}{\widetilde{{\rm cap}}}
\newcommand{\tCap}{\widetilde{{\rm Cap}}}
\newcommand{\ccap}{{\rm cap}}
\newcommand{\ucap}{\underline{\rm cap}}
\newcommand{\uCap}{\underline{\rm Cap}}
\newcommand{\cCap}{{\rm Cap}}
\newcommand{\ocap}{\overline{\rm cap}}
\newcommand{\oCap}{\overline{\rm Cap}}
\DeclareRobustCommand
{\mathringbig}[1]{\accentset{\smash{\raisebox{-0.1ex}{$\scriptstyle\circ$}}}{#1}\rule{0pt}{2.3ex}}
\newcommand{\cirH}{\mathringbig{H}}
\newcommand{\cirHs}{\mathringbig{H}{}^s}
\newcommand{\cirHt}{\mathringbig{H}{}^t}
\newcommand{\cirHm}{\mathringbig{H}{}^m}
\newcommand{\cirHzero}{\mathringbig{H}{}^0}
\newcommand{\deO}{{\partial\Omega}}
\newcommand{\OO}{{(\Omega)}}
\newcommand{\Rn}{{(\R^n)}}
\newcommand{\Id}{{\mathrm{Id}}}
\newcommand{\gap}{\mathrm{Gap}}
\newcommand{\ggap}{\mathrm{gap}}
\newcommand{\isom}{{\xrightarrow{\sim}}}
\newcommand{\half}{{1/2}}
\newcommand{\mhalf}{{-1/2}}
\newcommand{\inter}{{\mathrm{int}}}

\newcommand{\Hsp}{H^{s,p}}
\newcommand{\Htq}{H^{t,q}}
\newcommand{\tHsp}{{{\widetilde H}^{s,p}}}
\newcommand{\SP}{\ensuremath{(s,p)}}
\newcommand{\Xsp}{X^{s,p}}

\newcommand{\dd}{{d}}\newcommand{\pp}{{p_*}}

\newcommand{\Rnn}{\R^{n_1+n_2}}
\newcommand{\Tr}{{\mathrm{Tr}}}

\renewcommand{\Re}{\mathrm{Re}} 
\renewcommand{\Im}{\mathrm{Im}}
\newcommand{\sO}{\mathsf{O}}
\newcommand{\sC}{\mathsf{C}}
\newcommand{\sA}{\mathsf{A}}
\newcommand{\sM}{\mathsf{M}}
\newcommand{\sF}{\mathsf{F}}
\newcommand{\sG}{\mathsf{G}}
\newcommand{\mS}{\Gamma}
\newcommand{\omS}{{\overline{\mS}}}
\newcommand{\sumpm}[1]{\{\!\!\{#1\}\!\!\}}
\newcommand{\bH}{\mathbf{H}}
\newcommand{\bL}{\mathbf{L}}
\newcommand{\bu}{\mathbf{u}}
\newcommand{\cU}{\mathcal{U}}
\newcommand{\cK}{\mathcal{K}}
\newcommand{\cR}{\mathcal{R}}
\newcommand{\cX}{\mathcal{X}}
\newcommand{\cT}{\mathcal{T}}
\newcommand{\weakto}{\rightharpoonup}
\newtheorem{strat}[thm]{Strategy}
\newtheorem{res}[thm]{Results}
\title{A note on the convergence analysis of Laguerre approximations for analytic functions}
\author{Thomas Caussade${}^*$\\ David P. Hewett${}^*$\\[3mm]
${}^*$Department of Mathematics, University College London, London, United Kingdom}

\maketitle
\renewcommand{\thefootnote}{\arabic{footnote}}

\begin{abstract}
In a recent paper (H. Wang, Math.\ Comp.\ 93, 2861-2884, 2024), Wang has presented a number of results concerning the convergence of approximations based on generalised Laguerre polynomials, claiming to have provided ``the first rigorous proof of root-exponential convergence of Laguerre
approximations for analytic functions''. In this note we argue that the proofs of the main results of Wang's paper are incomplete, because they rely on taking a limit under an integral sign, and this is not properly justified. We explain how the proofs could be completed using the dominated convergence theorem, provided that certain nontrivial inequalities involving generalised Laguerre polynomials and their weighted Cauchy transforms hold. We make a related conjecture, which, if true, would complete the proofs of Wang's results. This conjecture remains unproven, but appears plausible, based on numerical experiments.

\end{abstract}


\section{Introduction}
\label{s:Intro}

The approximation of functions and integrals on a half-line using generalised Laguerre polynomials and related functions is a classical topic in approximation theory and numerical analysis. However, technical difficulties associated with the unbounded nature of the half-line have meant that the rigorous convergence theory for Laguerre approximations is less well developed than that for the analogous Legendre/Chebyshev/Jacobi approximations on bounded intervals.\footnote{For some historical background of Laguerre approximation theory we refer the reader to \cite{wang_laguerre_2024}.} 
In a recent paper \cite{wang_laguerre_2024}, 
Wang claims to present ``a comprehensive convergence rate analysis of Laguerre spectral approximations for analytic functions'', covering coefficient and error estimates for projection and interpolation, and error estimates for spectral differentiation and Gauss-Laguerre quadrature. For these problems, \cite{wang_laguerre_2024} claims to rigorously prove root-exponential convergence as the degree of the polynomials (or the number of quadrature points) increases, for functions analytic inside a parabola enclosing the half-line, and satisfying an algebraic growth condition at infinity. While we expect that the statements of the results in \cite{wang_laguerre_2024} may be correct, we contend that the proofs presented in \cite{wang_laguerre_2024} are incomplete, because they all rely on taking a non-uniform limit under an integral sign, which is nontrivial and not properly justified in \cite{wang_laguerre_2024}. 
Hence, we believe that providing a rigorous proof of root-exponential convergence of Laguerre
approximations for the general class of analytic functions  
considered in \cite{wang_laguerre_2024}
remains an open problem. 
In this note we explain the problem with the proofs in 
\cite{wang_laguerre_2024}, and state a conjecture (Conjecture \ref{Conj}), the proof of which would complete 
the analysis of \cite{wang_laguerre_2024}. 
\dhc{We also make some further observations on the convergence analysis of Gauss-Laguerre quadrature that may be of interest for applications.}

\section{Notation and terminology}
We start by briefly reviewing some notation and terminology from \cite{wang_laguerre_2024}. Given $n\in\N_0$, let $\mathbb{P}_n$ denote the space of polynomials of degree at most $n$, and $\mathbb{Q}_n:=\{\re^{-x/2}u(x):u\in\mathbb{P}_n\}$. 
Given $\alpha>-1$ (this will be our assumption on $\alpha$ throughout) and $n\in \N_0$ let $L^{(\alpha)}_n(z)\in \mathbb{P}_n$ denote the generalised Laguerre polynomial (GLP), which can be expressed as
\begin{align}
\label{e:GLPDef}
L^{(\alpha)}_n(z) = \frac{(\alpha+1)_n}{n!}M(-n,\alpha+1;-z), \qquad z\in\C, 
\end{align}
where $(\cdot)_n$ is the Pochhammer symbol and $M(\cdot,\cdot;\cdot)$ is the confluent hypergeometric function of the first kind (see \cite[\S13.2]{NIST:DLMF}). 
Let $\R_+:=[0,\infty)$, $\omega_\alpha(x):=x^\alpha \re^{-x}$ and 
$L^2_{\omega_\alpha}(\R_+):=\{f:\int_{\R_+}\omega_\alpha(x)|f(x)|^2\,\rd x<\infty\}$, and 
define the weighted Cauchy transform $\Phi^{(\alpha)}_n(z)$ by
\begin{align}
\label{e:PhiDef}
\Phi^{(\alpha)}_n(z):=\frac{1}{2\pi \ri}\int_{\R_+} \frac{\omega_\alpha(x)L^{(\alpha)}_n(x)}{z-x}\,\rd x, \qquad z\in \C\setminus\R_+.
\end{align}
We recall from \cite[Lemma 2.1]{wang_laguerre_2024} that $\Phi^{(\alpha)}_n(z)$ is analytic in the cut plane $\C\setminus\R_+$, satisfies 
\begin{align}
\label{e:PhiAsympt1}
\Phi^{(\alpha)}_n(z)=O(z^{-n-1}), \qquad z\to\infty,
\end{align}
and can be represented as
\begin{align}
\label{e:PhiRep}
\Phi^{(\alpha)}_n(z) = \frac{\ri}{2\pi}\Gamma(n+\alpha+1)U(n+1,1-\alpha;-z), \qquad z\in \C\setminus\R_+,
\end{align}
where $U(\cdot,\cdot;\cdot)$ is the confluent hypergeometric function of the second kind (see \cite[\S13.2]{NIST:DLMF}). 
A proof of \eqref{e:PhiRep} was presented in \cite[Lemma 2.1]{wang_laguerre_2024} using an application of L'H\^{o}pital's rule, the details of which 
were not provided. In Appendix \ref{s:PhiRep} we provide an alternative, more direct proof of \eqref{e:PhiRep}. 

Given $\rho>0$, let the parabola $P_\rho$ be defined by $P_\rho:=\{z\in\C:\, \Re[\sqrt{-z}\,]=\rho\}$, and let $D_\rho:=\{z\in\C:\, \Re[\sqrt{-z}\,]<\rho\}$ be the ``interior'' of this parabola, 
i.e.\ the open set containing $\R_+$ whose boundary is $P_\rho$. Here, and throughout this note, a fractional power of a complex variable is defined by its principal branch. 
See Figure \ref{f:Parabola} for an illustration of $P_\rho$ and $D_\rho$. 
In Cartesian form $P_\rho$ is described by the equation 
\begin{align}
\label{e:PrhoCart}
\Re[z]=\Im[z]^2/(4\rho^2)-\rho^2,
\end{align}
and, following \cite{wang_laguerre_2024}, we shall sometimes parametrise $P_\rho$ by
\begin{align}
\label{e:WangParam}
z(t)&=t^2-\rho^2-2t\rho \ri, \qquad t\in\R,
\end{align}
for which 
\begin{align}
|z(t)|=t^2+\rho^2,\qquad
z'(t)=2t-2\rho\ri,\qquad
|z'(t)|=2\sqrt{t^2+\rho^2}
,\qquad t\in\R.
\end{align}

\begin{figure}
\centering
\begin{center}
\begin{tikzpicture}[scale=8] %

    \fill[yellow!20] plot[domain=-0.6:0.6, samples=100] ({(\x)^2 - 0.09}, {0.6*\x}) 
        -- (0.45, 0.36) -- (0.45, -0.36) -- cycle;

    \draw[->, thick] (-0.2, 0) -- (0.5, 0) node[below] {$\text{Re}$};
    \draw[->, thick] (0, -0.38) -- (0, 0.38) node[right] {$\text{Im}$};

    \draw[thick] plot[domain=-0.6:0.6, samples=100] ({(\x)^2 - 0.09}, {0.6*\x});

    \fill (0,0) circle (0.2pt) node[below right, xshift=1pt] {$0$};
    \fill (-0.09, 0) circle (0.2pt) node[below left=2pt] {$-\rho^2$};
    \fill (0, 0.18) circle (0.2pt) node[left=2pt] {$+2i\rho^2$};
    \fill (0, -0.18) circle (0.2pt) node[left=2pt] {$-2i\rho^2$};

    \node at (0.25, 0.12) {$D_\rho$};
    \node at (0.20, -0.28) {$P_\rho$};

\end{tikzpicture}
\end{center}
\caption{The sets $P_\rho$ and $D_\rho$ \label{f:Parabola}}
\end{figure}

\section{Main results of \cite{wang_laguerre_2024} and proof strategy}
\label{s:Strategy}

The aim of \cite{wang_laguerre_2024} was to provide convergence results for certain quantities relating to approximation and quadrature using GLPs for a given function $f:\R_+\to\C$. 
The basic regularity assumption in \cite{wang_laguerre_2024} is that $f(z)$ (or $\re^{z/2}f(z)$ for certain results) satisfies the following condition:
\begin{align}
\label{e:fCond}
\text{$\exists\rho>0$ and 
$\exists\beta\in\R$ such that 
$f$ is analytic on $\overline{D_\rho}$ and $f(z)=O(|z|^\beta)$ as $z\to \infty$ in $\overline{D_\rho}$.}
\end{align}
With $n$ denoting the maximum degree of the approximating GLPs, or one minus the number of quadrature points in the Gauss-Laguerre quadrature, the main results of \cite{wang_laguerre_2024} all claim to show that a certain quantity $F_n(f)$ (depending on $f$ and $n$) decays root-exponentially as $n\to\infty$, with
\begin{align}
\label{e:FBound1}
\text{$F_n(f)\lessapprox Cn^\gamma\re^{-c\rho\sqrt{n}}\cK$,\,\, as $n\to\infty$,}
\end{align}
for some constants $C>0$, $\gamma\in \R$ and $c>0$, independent of $f$ and $n$, and an integral $\cK$ 
of the form
\begin{align}
\label{e:KDef}
\cK=
\int_{P_\rho}|\re^{-a z}||z|^b |f(z)|\,\rd s,
\end{align}
for some $a,b\in\R$ such that $\cK<\infty$.  
In \eqref{e:FBound1}, and henceforth, the notation ``$A_n\lessapprox C B_n$, as $n\to\infty$'' means that ``for all $C_0>C$ there exists $n_0\in\N$ such that $A_n\leq C_0B_n$ for all $n\geq n_0$''.

Specifically, our focus in this note is on the following key results from \cite{wang_laguerre_2024}.

\begin{res}
\label{r:WangRes}
The following results from \cite{wang_laguerre_2024} all claim  
bounds of the form \eqref{e:FBound1}:
\begin{itemize}
\item[(i)] \cite[Theorem 3.4]{wang_laguerre_2024}: 
Let $f$ satisfy \eqref{e:fCond}. 
Let $(\Pi^P_n f)(x):=\sum_{k=0}^n a^{(\alpha)}_k L^{(\alpha)}_k(x)$ denote the orthogonal projection of $f$ onto $\mathbb{P}_n$ in $L^2_{\omega_\alpha}(\R_+)$, with 
$a^{(\alpha)}_k=\frac{k!}{\Gamma(k+\alpha+1)}\int_{\R_+}\omega_\alpha(x)f(x)L^{(\alpha)}_k(x)\,\rd x$. 
Then 
\begin{align}
 \label{e:Ri1}
\text{$|a^{(\alpha)}_n|\lessapprox  \frac{1}{2\sqrt{\pi}}n^{-\alpha/2-1/4}\re^{-2\rho\sqrt{n}}
\int_{P_\rho}|\re^{-z/2}||z|^{\alpha/2-1/4}|f(z)|\,\rd s$,
\quad  as $n\to\infty$,}  
\end{align} 
and
\begin{align}
 \label{e:Ri2}
\text{$\|f-\Pi^P_nf\|_{\omega_\alpha}\lessapprox  \frac{1}{2\sqrt{2\pi\rho}} \re^{-2\rho\sqrt{n}}
\int_{P_\rho}|\re^{-z/2}||z|^{\alpha/2-1/4}|f(z)|\,\rd s$,
\quad as $n\to\infty$.} 
 \end{align}  
\item[(ii)] \cite[Theorem 4.2]{wang_laguerre_2024}: 
Let $f$ satisfy \eqref{e:fCond}. Let $p_n\in\mathbb{P}_n$ denote the interpolant of $f$ using the Laguerre points (the zeros of $L^{(\alpha)}_{n+1}(z)$) and $p^R_n\in\mathbb{P}_n$ the interpolant of $f$ using the Laguerre-Radau points (the zeros of $zL^{(\alpha+1)}_n(z)$). 
Then 
\begin{align}
 \label{e:Rii1} 
 \text{$\|f-p_n\|_{\omega_\alpha}\lessapprox \frac{1}{\sqrt{\pi}\,\rho^2}n^{1/4}\re^{-2\rho\sqrt{n}}
\int_{P_\rho}|\re^{-z/2 }||z|^{\alpha/2+1/4} |f(z)|\,\rd s$,
\quad as $n\to\infty$,} 
 \end{align}
and 
\begin{align}
 \label{e:Rii2} 
 \text{$\|f-p^R_n\|_{\omega_\alpha}\lessapprox \frac{\sqrt{2}}{\sqrt{\pi}\,\rho^2}n^{3/4}\re^{-2\rho\sqrt{n}}
\int_{P_\rho}|\re^{-z/2}||z|^{\alpha/2-1/4} |f(z)|\,\rd s$,
\quad as $n\to\infty$,} 
 \end{align}
\item[(iii)] \cite[Theorem 4.3]{wang_laguerre_2024}: 
Let $\re^{z/2}f(z)$ satisfy \eqref{e:fCond}. 
Let $q_n\in\mathbb{Q}_n$ denote the interpolant of $f$ using the Laguerre points (the zeros of $L^{(\alpha)}_{n+1}(z)$) and $p^R_n\in\mathbb{Q}_n$ the interpolant of $f$ using the Laguerre-Radau points (the zeros of $zL^{(\alpha+1)}_n(z)$). 
Then 
\begin{align}
 \label{e:Riii1} \text{$\|f-q_n\|_{\infty}\lessapprox \frac{\kappa_\alpha}{\sqrt{\pi}\rho^2}\,n^{|\alpha|/2+1/4}\re^{-2\rho\sqrt{n}} 
\int_{P_\rho}|z|^{\alpha/2+1/4} |f(z)|\,\rd s$,
\quad as $n\to\infty$,} 
\end{align}
and 
\begin{align}
 \label{e:Riii2} 
 \text{$\|f-q^R_n\|_{\infty}\lessapprox \frac{2\kappa_\alpha}{\sqrt{\pi}\rho^2}\, n^{|\alpha|/2+3/4}\re^{-2\rho\sqrt{n}}
\int_{P_\rho}|z|^{\alpha/2-1/4} |f(z)|\,\rd s$,
\quad as $n\to\infty$,} 
\end{align}
where 
$\kappa_\alpha:=1/\Gamma(\alpha+1)$ if $\alpha\geq 0$ and $\kappa_\alpha:=2$ if $\alpha\in(-1,0)$.
\footnote{Our bound \eqref{e:Riii2} in the case $\alpha\in(-1,0)$ differs by a factor of $2$ from that in \cite[Equation (4.8)]{wang_laguerre_2024}.}

\item[(iv)] \cite[Theorem 6.3]{wang_laguerre_2024}: 
Let $f$ satisfy \eqref{e:fCond}. Let $Q_n$ 
denote either the $(n+1)$-point Gauss-Laguerre quadrature approximation or the $(n+1)$-point Gauss-Laguerre-Radau quadrature approximation of the integral $I:=\int_{\R_+}\omega_\alpha(x)f(x)\,\rd x$. Then
\begin{align}
 \label{e:Riv}  \text{$|I-Q_n|\lessapprox \re^{-4\rho\sqrt{n}} 
\int_{P_\rho}|\re^{-z}||z|^{\alpha} |f(z)|\,\rd s$,
\quad as $n\to\infty$.} 
\end{align}
\end{itemize}
\end{res}

For each result in Results \ref{r:WangRes} the strategy in \cite{wang_laguerre_2024} for proving \eqref{e:FBound1} 
is as follows. 
\begin{strat}
\label{st:Strat}
The strategy in \cite{wang_laguerre_2024} for proving bounds of the form \eqref{e:FBound1} 
comprises four main steps: 
\begin{itemize}
\item[1:] Show that $F_n(f)$ satisfies 
\begin{align}
\label{e:FBound2}
F_n(f)
\lessapprox
Cn^{\gamma}
\re^{-c\rho\sqrt{n}}\int_{P_\rho} g_n(z)\,\rd s,
\quad \text{as }n\to\infty,
\end{align}
for 
some $C>0$, $\gamma\in\R$, $c>0$, and 
some sequence of nonnegative functions $g_n(z)$, such that $\int_{P_\rho} g_n(z)\,\rd s<\infty$ for sufficiently large $n$. 
\item[2:] Show that 
\begin{align}
\label{e:gng}
g_n(z)\to g(z), \qquad n\to \infty, 
\end{align}
for each $z\in P_\rho$, for some nonnegative function $g(z)$ such that $\int_{P_\rho} g(z)\,\rd s<\infty$. 
\item[3:] Deduce from Step 2 that
\begin{align}
\label{e:Intgng}
\int_{P_\rho} g_n(z)\,\rd s\to \int_{P_\rho} g(z)\,\rd s, \qquad n\to \infty. 
\end{align}
\item[4:] Combine Steps 1 and 3 to obtain \eqref{e:FBound1}, with 
$\cK=\int_{P_\rho} g(z)\,\rd s$. 
\end{itemize}
\end{strat}

One elementary thing that is missing from the proofs of the results in \cite{wang_laguerre_2024} listed 
in Results \ref{r:WangRes}(ii)-(iv)
relates to the justification of Step 1 in Strategy \ref{st:Strat}. For this step, the proofs of 
\cite[Theorems 4.2, 4.3 \& 6.3]{wang_laguerre_2024} all rely on \cite[Lemma 4.1]{wang_laguerre_2024}, which provides a formula for an interpolation error in terms of a contour integral over $P_\rho$. However, \cite[Lemma 4.1]{wang_laguerre_2024} only provides this formula for $f$ satisfying \eqref{e:fCond} with $\beta<1/2$. In Lemma \ref{l:intlag} in Appendix \ref{s:Lem4p1} we remedy this, by extending the statement of \cite[Lemma 4.1]{wang_laguerre_2024} to general $\beta\in\R$. 

The more serious omission from 
the proofs 
of all the 
results 
in 
\cite{wang_laguerre_2024} 
listed 
in Results \ref{r:WangRes}
is a rigorous justification 
of 
Step 3 in Strategy \ref{st:Strat}. As is well known, pointwise convergence \eqref{e:gng} of integrands is not in general sufficient to guarantee convergence \eqref{e:Intgng} of integrals %
(recall the classic ``escaping room'' counterexample). 
We believe that the absence of a rigorous justification of \eqref{e:Intgng} represents a nontrivial gap in the proofs presented in \cite{wang_laguerre_2024}.  

One sufficient condition for the validity of \eqref{e:Intgng} 
would be that the convergence \eqref{e:gng} held with $g_n(z)=g(z)(1+o(1))$, with $o(1)$ denoting a relative error that tends to zero uniformly on $P_\rho$ as $n\to \infty$. However, 
as we now explain, 
this does not hold for the 
cases 
in Results \ref{r:WangRes}.

\section{Pointwise convergence results}
\label{s:Pointwise}

The definitions of $g_n$ and $g$ for the results from \cite{wang_laguerre_2024} listed in Results \ref{r:WangRes} are given in Table \ref{t:gng}. 
(For brevity we do not reproduce the arguments justifying \eqref{e:FBound2} in each case - for this we refer the reader to \cite{wang_laguerre_2024}.) 
In each case, the validity of the pointwise convergence \eqref{e:gng} 
is justified using one or both of the following fundamental results:
\begin{itemize}
\item 
``Perron's formula'' for the large $n$ asymptotics of the GLPs (cited in \cite{wang_laguerre_2024} from \cite[Theorem 8.22.3]{szego1939orthogonal}), which implies that
\begin{align}
\label{e:Perron}
\frac{n^{\alpha/2-1/4}\re^{2\sqrt{-nz}}}{
2\sqrt{\pi}
L^{(\alpha)}_n(z)}\to 
\re^{-z/2}(-z)^{\alpha/2+1/4}, \qquad n\to\infty,\,z\in \C\setminus\R_+.
\end{align}
\item The large $n$ asymptotic behaviour of the Cauchy transform (stated as \cite[Formula (2.10)]{wang_laguerre_2024}, and derived by combining \eqref{e:PhiRep} with \cite[Equation (10.3.37)]{temme2014asymptotic} and \cite[Equation (10.40.2)]{NIST:DLMF}), viz.
\begin{align}
\label{e:PhiAsympt2}
-\ri 2\sqrt{\pi}
n^{-\alpha/2+1/4}\re^{2 \sqrt{-nz}}\Phi^{(\alpha)}_n(z)\to 
\re^{-z/2}(-z)^{\alpha/2-1/4}, \qquad n\to\infty, \,z\in \C\setminus\R_+.
\end{align}
\end{itemize}
Specifically, \cite[Theorem 3.4]{wang_laguerre_2024} uses \eqref{e:PhiAsympt2}; 
\cite[Theorem 4.2]{wang_laguerre_2024} uses 
\eqref{e:Perron}; 
\cite[Theorem 4.3]{wang_laguerre_2024} uses $\re^{z/2}$ times 
\eqref{e:Perron}; and
\cite[Theorem 6.3]{wang_laguerre_2024} uses the product of \eqref{e:Perron} and \eqref{e:PhiAsympt2}. In each case, one takes the modulus of the mentioned results, noting that $|\re^{2\sqrt{-nz}}|=\re^{2\Re\{\sqrt{-nz}\}}=\re^{2\rho\sqrt{n}}$ for $z\in P_\rho$. 

\begingroup 

\begin{table}[t]
\renewcommand{\arraystretch}{2.1}
\centering
\begin{tabular}{|c|c|c|}
\hline
Result & $g_n(z)$ & $g(z)$\\
\hline
\begin{tabular}{c}
\cite[Theorem 3.4]{wang_laguerre_2024} 
\\[-5mm]
i.e.\ \eqref{e:Ri1} and \eqref{e:Ri2}
\end{tabular}
& $
2\sqrt{\pi}
n^{-\alpha/2+1/4}\re^{2 \rho \sqrt{n}}|\Phi^{(\alpha)}_n(z)||f(z)|$ 
& 
$
|\re^{-z/2}||z|^{\alpha/2-1/4}|f(z)|$\\[2mm]
\hline

\begin{tabular}{c}
\cite[Theorem 4.2, Equation (4.2)]{wang_laguerre_2024} 
\\[-5mm]
i.e.\ \eqref{e:Rii1}
\end{tabular}
& $\dfrac{n^{\alpha/2-1/4}\re^{2 \rho \sqrt{n}}|f(z)|}{
2\sqrt{\pi}
|L^{(\alpha)}_{n+1}(z)|}$ 
& 
$
|\re^{-z/2}||z|^{\alpha/2+1/4}|f(z)|$\\[3mm]
\hline

\begin{tabular}{c}
\cite[Theorem 4.2, Equation (4.3)]{wang_laguerre_2024} 
\\[-5mm]
i.e.\ \eqref{e:Rii2}
\end{tabular}
& $\dfrac{n^{\alpha/2+1/4}\re^{2 \rho \sqrt{n}}|f(z)|}{2\sqrt{\pi}|zL^{(\alpha+1)}_{n}(z)|}$ 
& 
$
|\re^{-z/2}||z|^{\alpha/2-1/4}|f(z)|$\\[3mm]
\hline

\begin{tabular}{c}
\cite[Theorem 4.3, Equation (4.7)]{wang_laguerre_2024} 
\\[-5mm]
i.e.\ \eqref{e:Riii1}
\end{tabular}
& $\dfrac{n^{\alpha/2-1/4}\re^{2 \rho \sqrt{n}}|\re^{z/2}||f(z)|}{
2\sqrt{\pi}
|L^{(\alpha)}_{n+1}(z)|}$ 
& 
$
|z|^{\alpha/2+1/4}|f(z)|$
\\[3mm]
\hline

\begin{tabular}{c}
\cite[Theorem 4.3, Equation (4.8)]{wang_laguerre_2024} 
\\[-5mm]
i.e.\ \eqref{e:Riii2}
\end{tabular}
& $\dfrac{n^{\alpha/2+1/4}\re^{2 \rho \sqrt{n}}|\re^{z/2}||f(z)|}{
2\sqrt{\pi}
|zL^{(\alpha+1)}_n(z)|}$ 
& 
$
|z|^{\alpha/2-1/4}|f(z)|$
\\[3mm]
\hline
\begin{tabular}{c}
\cite[Theorem 6.3]{wang_laguerre_2024} 
\\[-5mm]
i.e.\ \eqref{e:Riv}
\end{tabular}
& $\dfrac{\re^{4 \rho \sqrt{n}}|\Phi^{(\alpha)}_{n+1}(z)||f(z)|}{|L^{(\alpha)}_{n+1}(z)|}$ 
& 
$|\re^{-z}||z|^{\alpha}|f(z)|$
\\[3mm]
\hline
\end{tabular}
\caption{The functions $g_n(z)$ and $g(z)$ for each of the results in Results 2.1. 
\label{t:gng}
}
\end{table}
\endgroup 

According to the sources from which they are derived, both \eqref{e:Perron} and \eqref{e:PhiAsympt2} hold pointwise in $\C\setminus\R_+$ and uniformly in compact subsets of $\C\setminus\R_+$. But it is clear that they cannot hold uniformly on unbounded subsets of $\C\setminus\R_+$ (such as $P_\rho$) because the exponential dependence on $z$ of the right-hand sides of \eqref{e:Perron} and \eqref{e:PhiAsympt2} is incompatible with the algebraic large $z$ behaviour of $L^{(\alpha)}_n(z)$ (which is a polynomial) and $\Phi^{(\alpha)}_n(z)$ (which decays like $|z|^{-n-1}$, see \eqref{e:PhiAsympt2}).\footnote{In \cite[Thm 8.22.3]{szego1939orthogonal} it is actually stated that the Perron formula leading to \eqref{e:Perron} holds with a relative error that tends to zero uniformly ``in any closed domain with no points in common with $z\geq 0$''. However, as we have explained, this cannot be true as it contradicts the polynomial nature of $L^{(\alpha)}_n(z)$. Presumably, 
Szeg\"o meant to say ``in any closed \textit{bounded} domain with no points in common with $z\geq 0$''.} 
Hence, for the results from \cite{wang_laguerre_2024} listed in Results \ref{r:WangRes}, \eqref{e:Intgng}  in Step 3 of Strategy \ref{st:Strat} is not justified by \eqref{e:Perron} and \eqref{e:PhiAsympt2} alone, and further justification is required to make the proofs complete.

\section{Completing the analysis of \cite{wang_laguerre_2024}}

In order to rigorously justify \eqref{e:Intgng} 
one could appeal to the dominated convergence theorem. For this, one would need to prove the existence of a nonnegative function $h(z)$ 
and an $n_0\in \N$ such that
\begin{align}
\label{e:hConds}
\int_{P_\rho}h(z)\, \rd s<\infty \qquad \text{and} \qquad g_n(z)\leq h(z), \,z\in P_\rho, \, n\geq n_0. 
\end{align}

The following theorem provides sufficient conditions under which \eqref{e:hConds} holds for the problems from \cite{wang_laguerre_2024} listed in Results \ref{r:WangRes}. 
We note that if both \eqref{e:ConjPhi} and \eqref{e:ConjL} hold then \eqref{e:ConjRatio} holds automatically.

\begin{thm}
\label{t:Ass}
Fix $\rho>0$ and $\alpha>-1$. 
\begin{itemize}
\item[(i)]
Suppose that 
\begin{align}
\label{e:ConjPhi}
\forall\gamma\in\R \,\,\exists n_0\in\N \,\,\text{ such that }\,\,\sup_{n\geq n_0}\sup_{z\in P_\rho} n^{-\alpha/2+1/4}\re^{2\rho\sqrt{n}}|\Phi^{(\alpha)}_{n}(z)||z|^{\gamma}<\infty.
\end{align}
Then \cite[Theorem 3.4]{wang_laguerre_2024} (cf.\ Results \ref{r:WangRes}(i)) holds.
\item[(ii)]
Suppose that 
\begin{align}
\label{e:ConjL}
\forall\gamma\in\R \,\,\exists n_0\in\N \,\,\text{ such that }\,\,\sup_{n\geq n_0}\sup_{z\in P_\rho} \frac{n^{\alpha/2-1/4}\re^{2\rho\sqrt{n}}|z|^{\gamma}}{|L^{(\alpha)}_{n}(z)|}<\infty.
\end{align}
Then \cite[Theorems 4.2 and 4.3]{wang_laguerre_2024} (cf.\ Results \ref{r:WangRes}(ii)-(iii)) hold.
\item[(iii)]
Suppose that 
\begin{align}
\label{e:ConjRatio}
\forall\gamma\in\R \,\,\exists n_0\in\N \,\,\text{ such that }\,\,\sup_{n\geq n_0}\sup_{z\in P_\rho} \frac{\re^{4\rho\sqrt{n}}|\Phi^{(\alpha)}_{n}(z)||z|^{\gamma}}{|L^{(\alpha)}_{n}(z)|}<\infty.
\end{align}
Then \cite[Theorem 6.3]{wang_laguerre_2024} (cf.\ Results \ref{r:WangRes}(iv)) holds.
\end{itemize}
\end{thm}
\begin{proof}
As explained above, in order to complete the proofs of the mentioned results from \cite{wang_laguerre_2024} using the dominated convergence theorem we need to show that, for each corresponding choice of $g_n(z)$ in Table \ref{t:gng}, there exists a nonnegative function $h(z)$ and an $n_0\in \N$ satisfying \eqref{e:hConds}. 

Let us first prove part (i), which relates to \cite[Theorem 3.4]{wang_laguerre_2024} (cf.\ Results \ref{r:WangRes}(i)). As in the hypothesis of \cite[Theorem 3.4]{wang_laguerre_2024}, suppose that $f$ satisfies 
\eqref{e:fCond} for some $\beta\in\R$. Then there exists $C>0$ such that $|f(z)|\leq C |z|^{\beta}$ for $z\in P_\rho$. If \eqref{e:ConjPhi} holds, then, taking $\gamma=\beta+2$, there exist $n_0\in\N$ and  $C_\infty>0$ such that the function $g_n(z)$ in Table \ref{t:gng} corresponding to \cite[Theorem 3.4]{wang_laguerre_2024} satisfies $|g_n(z)|\leq C C_\infty |z|^{-2}$ for all $n\geq n_0$ and $z\in P_\rho$. Noting that $\int_{P_\rho}|z|^{-2}\,\rd s<\infty$, we see that \eqref{e:hConds} holds with $h(z):=C C_\infty |z|^{-2}$. 

The proofs of parts (ii) and (iii) follow similarly, noting that, in part (ii), when proving \cite[Theorem 4.3]{wang_laguerre_2024}, one needs to modify the assumption on $f$, to assume that $\re^{z/2}f(z)$ satisfies \eqref{e:fCond}. 
\end{proof}

As yet we have been unable to rigorously prove the conditions \eqref{e:ConjPhi}-\eqref{e:ConjRatio} for any choice of $\rho>0$ and $\alpha>-1$. However, informed by numerical experimentation we make the following conjecture. 
We note the second part of the conjecture is stronger than the first, since  \eqref{e:ConjPhi2}, \eqref{e:ConjL2} and \eqref{e:ConjRatio2} imply  \eqref{e:ConjPhi}, \eqref{e:ConjL} and \eqref{e:ConjRatio} with $n_0=\lceil\gamma -1 \rceil$, $n_0=\lceil\gamma \rceil$ and $n_0= \lceil(\gamma-1)/2 \rceil$, respectively.

\begin{conj}
\label{Conj}
We conjecture that \eqref{e:ConjPhi}-\eqref{e:ConjRatio} hold for every $\rho>0$ and $\alpha>-1$. 
Furthermore, we conjecture that, 
for every $\rho>0$ and $\alpha>-1$, and for every $n_0\in\N_0$, it holds that 
\begin{align}
\label{e:ConjPhi2}
\sup_{n\geq n_0}\sup_{z\in P_\rho} n^{-\alpha/2+1/4}\re^{2\rho\sqrt{n}}|\Phi^{(\alpha)}_{n}(z)||z|^{n_0+1}<\infty,
\end{align}
\begin{align}
\label{e:ConjL2}
\sup_{n\geq n_0}\sup_{z\in P_\rho} \frac{n^{\alpha/2-1/4}\re^{2\rho\sqrt{n}}|z|^{n_0}}{|L^{(\alpha)}_{n}(z)|}<\infty,
\end{align}
and
\begin{align}
\label{e:ConjRatio2}
\sup_{n\geq n_0}\sup_{z\in P_\rho} \frac{\re^{4\rho\sqrt{n}}|\Phi^{(\alpha)}_{n}(z)||z|^{2n_0+1}}{|L^{(\alpha)}_{n}(z)|}<\infty.
\end{align}
\end{conj}

\begin{rem}
\label{rem:Cong}
Recalling that 
$|\Phi^{(\alpha)}_{n}(z)|=O(|z|^{-n-1})$ and $1/|L^{(\alpha)}_{n}(z)|=O(|z|^{-n})$ as $|z|\to\infty$,
we note that, 
for each of \eqref{e:ConjPhi2}-\eqref{e:ConjRatio2}, for every $n_0\in\N$ the inner supremum over $z\in P_\rho$ is guaranteed to be finite for each fixed $n\geq n_0$. The claim of the second part of Conjecture \ref{Conj} is that in each case one can also take an outer supremum over $n\geq n_0$ and obtain a finite quantity, in spite of the root-exponential prefactor. 
\end{rem}

\begin{figure}[t!]
\centering
\subfloat[$\alpha=0$, $\rho=1$, $n_0=1$]{
\includegraphics[width=.48\textwidth]{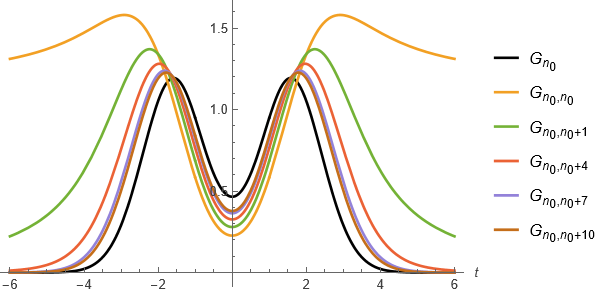}
\hspace{3mm}
\includegraphics[width=.48\textwidth]{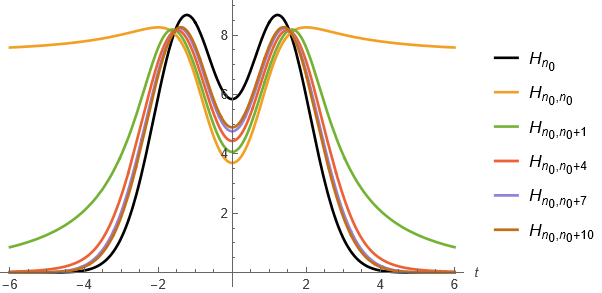}}

\subfloat[$\alpha=1$, $\rho=0.1$, $n_0=5$]{
\includegraphics[width=.48\textwidth]{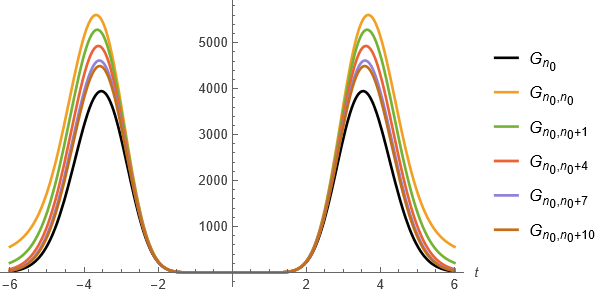}
\hspace{3mm}
\includegraphics[width=.48\textwidth]{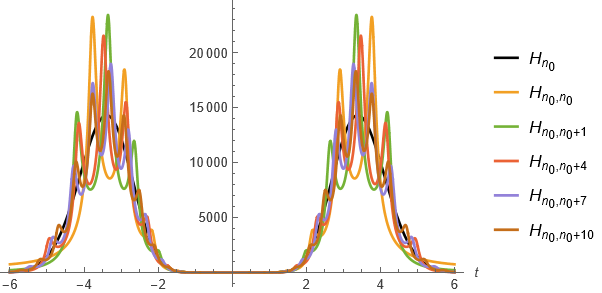}}

\subfloat[$\alpha=-0.5$, $\rho=2$, $n_0=3$]{
\includegraphics[width=.48\textwidth]{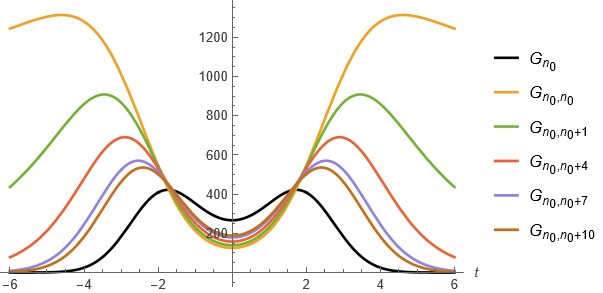}
\hspace{3mm}
\includegraphics[width=.48\textwidth]{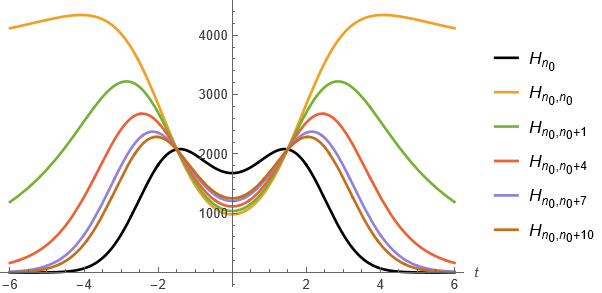}}

\caption{
Plots of $G_{n_0}(z(t))$ and $H_{n_0}(z(t))$ and $G_{n_0,n}(z(t))$ and $H_{n_0,n}(z(t))$ for different values of $\rho$, $\alpha$, $n_0$ and $n$.
\label{f:Conj}
}
\end{figure}

Some numerical evidence in support of Conjecture \ref{Conj} is provided in Figure \ref{f:Conj}. The results presented here are just a sample of many that we have computed for different parameter values, all of which appear to support Conjecture \ref{Conj}. We focus here on results relating to \eqref{e:ConjPhi2} and \eqref{e:ConjL2}, since these two together imply \eqref{e:ConjRatio2}. 
In Figure \ref{f:Conj} we show plots of
\begin{align}
\label{}
G_{n_0,n}(z):=n^{-\alpha/2+1/4}\re^{2\rho\sqrt{n}}|\Phi^{(\alpha)}_n(z)||z|^{n_0+1}
\end{align}
and
\begin{align}
\label{}
H_{n_0,n}(z):=\frac{n^{\alpha/2-1/4}\re^{2\rho\sqrt{n}}|z|^{n_0}}{|L^{(\alpha)}_n(z)|}
\end{align}
for different values of $\rho$, $\alpha$, $n_0$, and $n$, for points $z$ on the part of $P_\rho$ near its apex, i.e.\ for $z=z(t)$ given by \eqref{e:WangParam}, 
for $t$ in an interval containing $0$.
For reference we also plot the functions 
\begin{align}
\label{}
G_{n_0}(z):=\frac{1}{2\sqrt{\pi}}|\re^{-z/2}||z|^{\alpha/2-1/4+n_0+1}
\end{align}
and
\begin{align}
\label{}
H_{n_0}(z):=2\sqrt{\pi}|\re^{-z/2}||z|^{\alpha/2+1/4+n_0},
\end{align}
which are the pointwise limits of $G_{n_0,n}(z)$ and $H_{n_0,n}(z)$ as $n\to\infty$ for fixed $z\in P_\rho$. 
The results in Figure \ref{f:Conj} appear to be consistent with our conjectures \eqref{e:ConjPhi2} and \eqref{e:ConjL2}, which can be written as
\begin{align}
\label{e:ConjPhi3}
\sup_{n\geq n_0}\sup_{z\in P_\rho} G_{n,n_0}(z) <\infty
\end{align}
and 
\begin{align}
\label{e:ConjL3}
\sup_{n\geq n_0}\sup_{z\in P_\rho} H_{n,n_0}(z) <\infty.
\end{align}
However, rigorous justification of these conjectures requires delicate and apparently nontrivial estimation of the special functions $L^{(\alpha)}_n(z)$ and $\Phi^{(\alpha)}_n(z)$ (recall \eqref{e:GLPDef} and \eqref{e:PhiRep}), which we have not yet been able to carry out, and which we leave for future work.

\section{\dhc{Further results on Gauss-Laguerre quadrature}}

The aim of \cite{wang_laguerre_2024} was to present error estimates that are valid for general functions $f$ satisfying a regularity condition of the form \eqref{e:fCond}. 
\dhc{For specific examples 
it is possible to derive sharper estimates.  
Indeed, for meromorphic functions with a finite number of poles and algebraic behaviour at infinity, 
rigorous convergence estimates of the form \eqref{e:FBound1} for the problems considered in Results \ref{r:WangRes} can be derived 
by expressing the contour integral representation of the corresponding quantity $F_n(f)$ 
as a sum of residues, to which the pointwise asymptotics \eqref{e:Perron} and \eqref{e:PhiAsympt2} can be applied. 
For functions with branch points/cuts, under appropriate assumptions it is possible to deform the integration contour onto the branch cuts before applying \eqref{e:Perron} and \eqref{e:PhiAsympt2} to the resulting integral bound for the error. When the branch cuts are finite in extent, and do not intersect $\R_+$, this immediately leads to a rigorous error estimate. When the branch cuts are infinite in extent one needs to assume Conjecture \ref{Conj} (or something analogous to it) to justify taking the limit under the integral sign. }

We illustrate these ideas briefly for the case of Gauss-Laguerre quadrature for the function 
\dhc{
\begin{align}
\label{e:fExamp}
f(z)=(z-z_0)^\beta,
\end{align}
where $z_0\in\C\setminus\R_+$ and 
$\beta\in[-1,\infty)$, $\beta\not\in \N_0$. 
In the case $\beta\neq -1$ we assume the branch cut in \eqref{e:fExamp} has been chosen so that it does not intersect $\R_+$.}
Let
\begin{align}
\label{e:IDef}
I:=\int_0^\infty x^{\alpha}\re^{-x}f(x)\,\rd x
\end{align} and let $Q_n$ denote the $(n+1)$-point Gauss-Laguerre quadrature approximation to $I$. 
Arguing as in \cite[Proof of Theorem 6.3]{wang_laguerre_2024} (using \cite[Lemma 4.1]{wang_laguerre_2024}, cf.\ Lemma \ref{l:intlag} below) we have, for sufficiently small $\rho>0$, and for every $n\in\N$ 
\dhc{satisfying $n>\beta-1$}, 
that 
\begin{align}
\label{e:QuadErr}
I-Q_n = \int_{P_\rho}\frac{\Phi^{(\alpha)}_{n+1}(z)}{L^{(\alpha)}_{n+1}(z)}f(z)\,\rd z.
\end{align}
\dhc{Note that in the case $\beta\neq -1$ before deforming onto $P_\rho$ one may need to move the branch cut in \eqref{e:fExamp} so that it lies outside $P_\rho$. }

\subsection{\dhc{Meromorphic functions}}
\label{s:mero}

\dhc{We first consider the case $\beta=-1$, in which $f(z)=(z-z_0)^{-1}$ has a single simple pole at $z_0$.}
Since $f(z)=O(|z|^{-1})$ as $z\to\infty$, \dhc{\eqref{e:QuadErr} holds for every $n\in\N$.}
Furthermore, the fact that $f(z)=O(|z|^{-1})$ as $z\to\infty$ 
means that the contour in \eqref{e:QuadErr} can be deformed across the pole at $z=z_0$ and off to infinity (recalling that $\frac{\Phi^{(\alpha)}_{n+1}(z)}{L^{(\alpha)}_{n+1}(z)}=O(|z|^{-2n-3})$ as $z\to\infty$), leaving only the residue contribution. The latter can be approximated using a pointwise application of the asymptotics \eqref{e:Perron} and \eqref{e:PhiAsympt2} to give
\begin{align}
 \label{e:QuadErr2}
 |I-Q_n| = 2\pi \left|\frac{\Phi^{(\alpha)}_{n+1}(z_0)}{L^{(\alpha)}_{n+1}(z_0)}\right| \sim 2\pi |z_0|^ \alpha \re^{-4\Re[\sqrt{-z_0}\,]\sqrt{n}} \re^{-\Re[z_0]}, \qquad n\to\infty. 
 \end{align} 
We note that a similar formula was presented in \cite[Equation (2.6)]{barrett1961quad} for the case of a pair of complex conjugate simple poles (cf.\ \cite[Remark 6.4]{wang_laguerre_2024}).

Contour plots of the left- and right-hand sides of \eqref{e:QuadErr2} in the case $\alpha=0$ for two different $n$ values are presented in Figure \ref{f:QuadErr}; they show excellent agreement, especially for the larger value of $n$, provided $z_0$ is not too close to the positive real axis (we note in particular that the quadrature formula $Q_n$ blows up as $z_0$ approaches one of the quadrature points).
\dhc{For $\alpha=0$} the equi-error curves for the right-hand side of \eqref{e:QuadErr2} are the contours of the function 
$4\Re[\sqrt{-z_0}\,]\sqrt{n}+\Re[z_0]$.
For fixed $n$ they are tear-drop shaped, and they are only approximately parabolic ($\Re[\sqrt{-z_0}\,]=const$) in the regime where $\Re[\sqrt{-z_0}\,]\sqrt{n}\gg \Re[z_0]$. %
The deviation from purely parabolic equi-error curves is of course a consequence of the exponential weight in the integral 
\eqref{e:IDef}. The practical implication of \eqref{e:QuadErr2} is that, for integrands with a pole $z_0\in\C\setminus \R_+$ for which $\Re[\sqrt{-z_0}\,]$ (which represents the parameter $\rho_0$ of the parabola $P_{\rho_0}$ passing through $z_0$) is small, 
while the asymptotic rate of convergence ($|I-Q_n|=O(\re^{-4\Re[\sqrt{-z_0}\,]\sqrt{n}})$ as $n\to\infty$) may be slow, the actual quadrature error may be acceptably small if $\Re[z_0]$ is sufficiently large, because of the exponential factor $\re^{-\Re[z_0]}$ in \eqref{e:QuadErr2}. 
\dhc{We shall return to this discussion in \S\ref{s:unif}.}

\begin{figure}[t!]
\centering
\subfloat[]{%
\includegraphics[width=.75\textwidth]%
{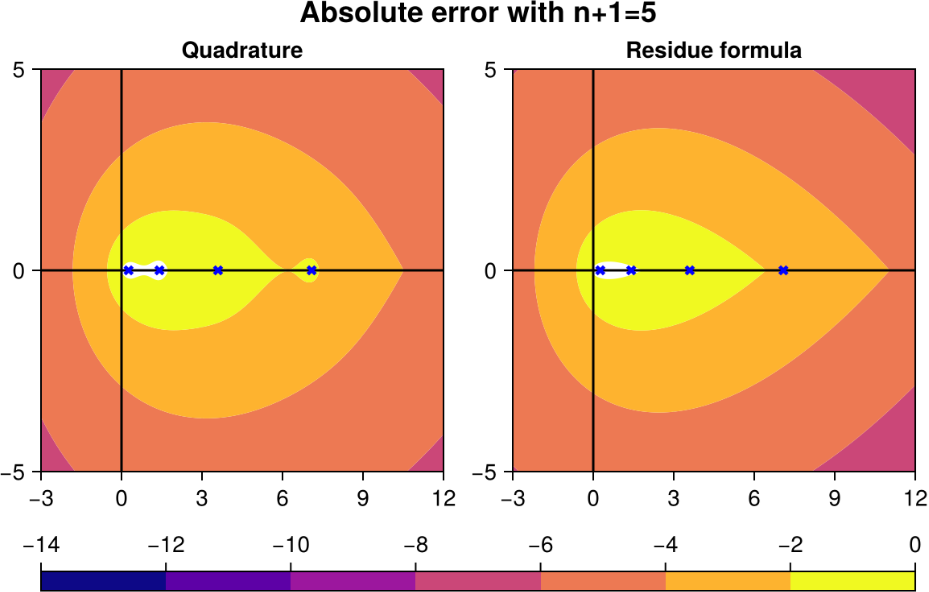}}

\vspace{3mm}
\subfloat[]{%
\includegraphics[width=.75\textwidth]%
{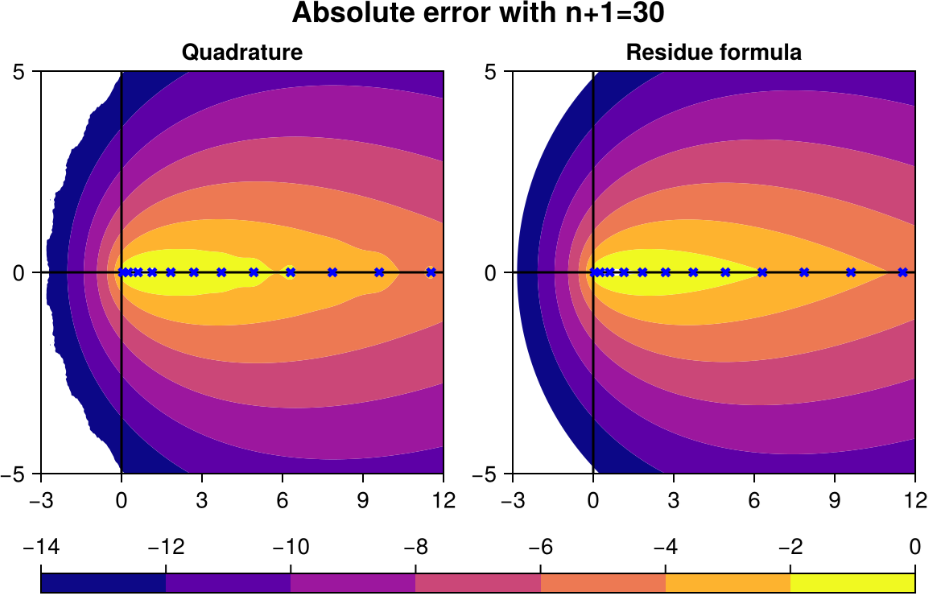}}
\caption{
Contour plots of the left- and right-hand side of \eqref{e:QuadErr2}, as a function of the singularity location $z_0$, showing the absolute error\dhc{, in $\log_{10}$ scale,} in Gauss-Laguerre quadrature with $\alpha=0$ and $f=(z-z_0)^{-1}$ 
for (a) $n+1=5$ points and (b) $n+1=30$ points.
The quadrature points lying within the plot windows are shown by blue crosses.
\label{f:QuadErr}
}
\end{figure}

\subsection{Functions with branch points}
\label{s:branch}

\dhc{We now consider the case $\beta > -1$, $\beta\not\in\N_0$, so that $f(z)=(z-z_0)^\beta$ has a branch point at $z_0$. 
Without loss of generality, we assume that $\Im[z_0]\geq 0$. 
Let $\rho_0=\Re[\sqrt{-z_0}]$ and let $P_{\rho_0}$ denote the Laguerre parabola passing through $z_0$. Let $\gamma:=\{z\in P_{\rho_0}: \Im[z]\geq \Im[z_0]\}$ denote the semi-infinite portion of $P_{\rho_0}$ starting at $z_0$ and going to infinity as illustrated in Figure \ref{f:contbc}. We take the branch cut for $f(z)$ to lie on $\gamma$. Then \eqref{e:QuadErr} is valid for any $0<\rho<\rho_0$. Furthermore, for sufficiently large $n$ the decay of the integrand at infinity, and the fact that $\beta>-1$, means that the contour in \eqref{e:QuadErr} can be deformed onto a contour from infinity to $z_0$ along the top side of $\gamma$ and back to infinity along the bottom side of $\gamma$ (the blue contour in Figure \ref{f:contbc}). 
Assuming that Conjecture~\ref{Conj} holds, the argument used in the proof of Theorem \ref{t:Ass} allows us to apply the asymptotics \eqref{e:Perron} and \eqref{e:PhiAsympt2} inside the integral along each side of $\gamma$, with the result that}
\begin{equation}
\label{e:brancherr}
	|I-Q_n| \lessapprox 2\re^{-4\Re[\sqrt{-z_0}]\sqrt{n}} \int_{\gamma}|\re^{-z}||z^{\alpha}| |f(z)| \,\rd s,
	\end{equation}
\dhc{after noting that $|f(z)|=|z-z_0|^\beta$ is the same on both sides of $\gamma$. 
The advantage of \eqref{e:brancherr} over \eqref{e:Riv}	is that the integration is now only over $\gamma$ rather than over a full Laguerre parabola. This allows us to obtain estimates that are sharper in their $z_0$-dependence than one can obtain from \eqref{e:Riv}.
Indeed, assuming that Conjecture~\ref{Conj} holds, combining \eqref{e:brancherr} with Lemma \ref{l:brb} we have that
	\begin{equation}
	\label{e:Boundzmu}
	|I-Q_n|\lessapprox C \re^{-\Re[z_0]}\re^{-4\Re[\sqrt{-z_0}]\sqrt{n}}
		(\chi_{\mu\geq 0}+|z_0|^\mu)
		\qquad
		\mu:=\beta/2+\alpha+1/2,
	\end{equation}
	for some constant $C>0$ depending only on $\alpha$ and $\beta$. 
}

\begin{figure}[t!]
\centering
\begin{tikzpicture}[scale=4]
    \draw[->, thick] (-0.2, 0) -- (1, 0) node[below] {$\text{Re}$};
    \draw[->, thick] (0, -0.6) -- (0, 0.6) node[right] {$\text{Im}$};

    \draw[thick, dashed] plot[domain=-1.:1., samples=100] ({(\x)^2 - 0.09}, {0.6*\x});

    \draw[very thick, red] plot[domain=0.4359:1., samples=50] ({(\x)^2 - 0.09}, {0.6*\x});
    
	\draw[blue] plot[domain=0.4359:1.0, samples=50] ({(\x)^2 - 0.09}, {0.6*\x + 0.04});
    \draw[blue] plot[domain=1.0:0.4359, samples=50] ({(\x)^2 - 0.09}, {0.6*\x - 0.04});
    \draw[blue] (0.1, {0.6*0.4359 + 0.04}) arc (90:+270:0.04);

    \fill (0.1, {0.6*0.4359}) circle (0.4pt) node[below right] {$z_0$};

    \node at (0.38, -0.28) {$P_{\rho_0}$};
       \node[below, red] at ({0.98}, {0.68}) {$\gamma$};
\end{tikzpicture}
\caption{Contour depiction for quadrature error estimation of a function with a branch point at $z_0$. The contour 
$\gamma$ is a portion of the Laguerre parabola $P_{\rho_0}$ with parameter $\rho_0 = \Re[\sqrt{-z_0}]$. 
}
\label{f:contbc}
\end{figure}

\dhc{

\subsection{Uniformity with respect to $z_0$}
\label{s:unif}

For the examples considered in \S\ref{s:mero} and \S\ref{s:branch}, where the integrand has a singularity at $z_0\in \C\setminus\R_+$, we have seen that, up to an algebraic factor, the $z_0$-dependence of our quadrature error bounds is governed by the quantity 
\begin{equation}\label{e:Ez0}
	E(z_0,n) := \re^{-\Re[z_0]}  \re^{-4\Re[\sqrt{-z_0}]\sqrt{n}}.
\end{equation}
It is natural to ask whether one can obtain uniform bounds on $E(z_0,n)$ inside certain subsets of $\C$ excluding the positive real axis. 
Given $\sigma\in\R$, let $H_{\sigma}:=\{z\in\C:\Re[z]\geq\sigma\}$. Then, directly from \eqref{e:Ez0}, given $\rho>0$, %
in $H_{\sigma}\setminus D_{\rho}$ 
we have the uniform square-root-exponential bound
\begin{align}
\label{e:Eout}
E(z_0,n)\leq \re^{-\sigma} \re^{-4\rho\sqrt{n}}, \qquad z_0\in H_{\sigma}\setminus D_{\rho}, \,n\in\N.
\end{align}
The following result shows that a uniform bound on $E(z_0,n)$ 
can be derived for a different subset of $H_{\sigma}$, which intersects $D_{\rho}$ nontrivially, if one is willing to accept cube-root-exponential convergence instead of square-root-exponential convergence. 

\begin{prop}
\label{p:unif}
Let $E(z_0,n)$ be defined as in \eqref{e:Ez0}, and let $\tau_0>0$. 
Then 
\begin{equation}\label{e:Eab}
	E(z_0, n) \leq  \re^{\tau_0/2}\re^{-3\tau_0^{2/3} n^{1/3}}, 
\qquad \Re[z_0]\geq 0, \, |\Im[z_0]|\geq \tau_0,\,n\geq \tau_0/8.
\end{equation}
\end{prop}

\begin{proof}
By symmetry 
it suffices to consider $\Im[z_0]\geq 0$. So, let $z_0=\sigma +\ri \tau$, with $\sigma \geq 0$ and $\tau\geq 0$. Then 
\[ E(z_0,n) = \re^{-\phi(\sigma ,\tau,n)},\]
where 
\[ \phi(\sigma ,\tau,n) = \sigma  + 2\sqrt{2}\sqrt{n}\sqrt{\sqrt{\tau^2+\sigma ^2}-\sigma }.
\]
For fixed $s$ and $n$, $\phi(\sigma,\tau,n)$ is an increasing function of $\tau$, so for $\tau\geq \tau_0$ we have 
\begin{align*}
\label{}
\phi(\sigma ,\tau,n)\geq \phi(\sigma ,\tau_0,n) = \sigma  + 2\sqrt{2}\sqrt{n}\sqrt{\sqrt{\tau_0^2+\sigma ^2}-\sigma }
&= \sigma + \frac{2\sqrt{2}\tau_0\sqrt{n}}{\sqrt{\sqrt{\tau_0^2+\sigma ^2} + \sigma }}
\geq 
\sigma  + \frac{2\sqrt{2}\tau_0\sqrt{n}}{\sqrt{\tau_0+2\sigma }}.
\end{align*}
Let $g(\sigma ):=\sigma  + \frac{2\sqrt{2}\tau_0\sqrt{n}}{\sqrt{\tau_0+2\sigma }}$. 
Then $g'(\sigma ) = 1 - \frac{2\sqrt{2}\sqrt{n}\tau_0}{(\tau_0+2\sigma )^{3/2}}$, which vanishes at 
$\sigma _*:=\tau_0^{2/3}n^{1/3}-\tau_0/2$. 
When $n\geq \tau_0/8$, 
we have $\sigma _*>0$ and $\min_{\sigma \geq 0}g(\sigma ) = g(\sigma _*)
=3\tau_0^{2/3}n^{1/3}-\tau_0/2$, 
from which \eqref{e:Eab} follows.

\end{proof}

Combining \eqref{e:Eout} and \eqref{e:Eab} (for $\tau_0=2\rho^2$) with 
\eqref{e:QuadErr2} and \eqref{e:Boundzmu}  
we can obtain uniform bounds for the Gauss-Laguerre quadrature error for the examples considered in \S\ref{s:mero} and \S\ref{s:branch}. The statement of Corollary \ref{c:GauLag} is illustrated schematically in Figure~\ref{f:unifregions}.

\begin{cor}
\label{c:GauLag}
Let $\rho>0$, $\sigma\in\R$. %
Then the Gauss-Laguerre quadrature error for the integrand $f(z)=(z-z_0)^{-1}$ satisfies 
\begin{align}
\label{e:QuadFinal}
|I-Q_n| \lessapprox 
2\pi |z_0|^\alpha 
\begin{cases}
\re^{-\sigma}\re^{-4\rho\sqrt{n}}, & z_0\in 
H_{\sigma}\setminus D_{\rho},\\
\re^{\rho^2}
\re^{-3\cdot 2^{2/3}\rho^{4/3}n^{1/3}}, 
& z_0\in 
D_{\rho},\ |\Im[z_0]|\geq 2\rho^2.
\end{cases}
\end{align}
If $\beta>-1$, $\beta\not\in\N_0$, and if Conjecture \ref{Conj} holds, then the Gauss-Laguerre quadrature error for the integrand $f(z)=(z-z_0)^{\beta}$ also satisfies \eqref{e:QuadFinal}, but with the factor $2\pi |z_0|^\alpha $ replaced by $C(\chi_{\mu\geq 0}+|z_0|^\mu)$, where $\mu=\beta/2+\alpha+1/2$ and $C>0$ depends only on $\alpha$ and $\beta$. 
\end{cor}

In Figure \ref{f:unif} we plot the quadrature error against $n$ for the case $\alpha=0$, $\beta=-1/2$ for various points on the line $\Im[z_0]=1$. The envelope of the error curves for different $z_0$ indeed exhibits the cube-root-exponential decay predicted by Corollary \ref{c:GauLag}. 
}

\begin{figure}[t!]
	\centering
	\begin{tikzpicture}[scale=1.2]
        \node[inner sep=0, anchor=south west] (image) at (0,0) {\includegraphics[height = 6.6cm]{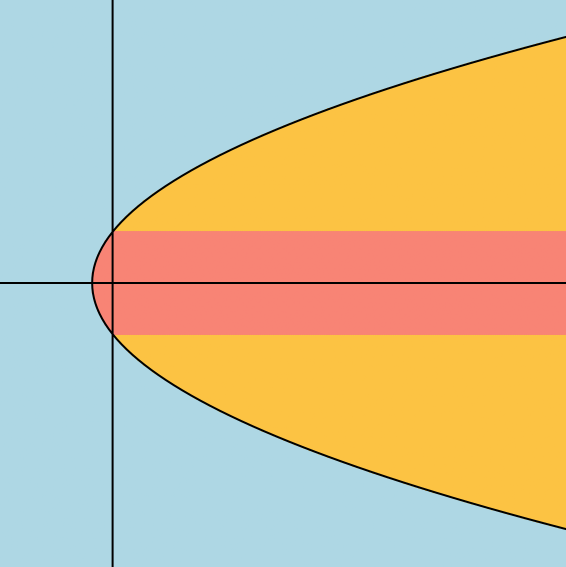}};
         \node at ({2.08}, {0.4}) 
         {$|I-Q_n|= O(\re^{-4\rho\sqrt{n}})$};
         \node at ({3.8}, {3.7}) 
{$|I-Q_n|=O(\re^{-cn^{1/3}})$};
         \fill (1.1, {3.1+0.18}) circle (2pt) node[above left] {$2\ri \rho^2$};
         \fill (0, {2.765}) circle (2pt) node[below right]
{$\sigma$};
         \node at ({2.85}, {4.7}) 
{$P_{\rho}$};
         \node at ({4.3}, {1.6}) 
{$D_{\rho}$};
    \end{tikzpicture}
	\caption{Schematic illustration of the result of Corollary \ref{c:GauLag}, showing the regions of the $z_0$ plane in which we have proved uniform error bounds (as $n\to\infty$) for Gauss-Laguerre quadrature applied to~\eqref{e:IDef} with $f(z)=(z-z_0)^\beta$, $\beta\geq -1$, $\beta\not\in\N_0$ (modulo a factor depending on $z_0$ algebraically). If $z_0$ lies in the blue region, that is, outside the Laguerre parabola with parameter $\rho_0$ and to the right of the line $\Re[z_0]=\sigma$, we have uniform square-root-exponential convergence. If $z_0$ is in the orange region, that is, between the parabola and the lines $\Im[z_0]=\pm 2\rho^2$, we have uniform cube-root-exponential convergence (with rate $c=3\cdot 2^{2/3}\rho^{4/3}$). 
}
	\label{f:unifregions}
\end{figure}

\begin{figure}[t!]
	\centering
	\centering
	\subfloat[$\log_{10}$ vs. square-root]{
	\includegraphics[height = 72mm]{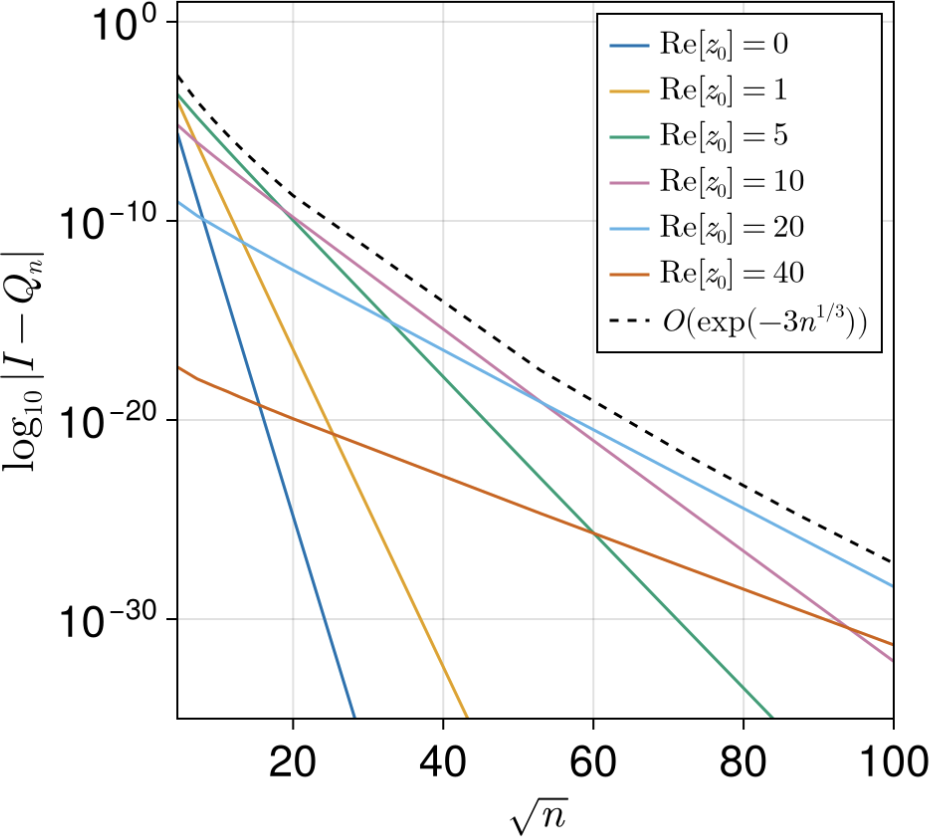} }
	\hspace{1mm}
	\subfloat[$\log_{10}$ vs. cube-root]{
	\includegraphics[height = 72mm]{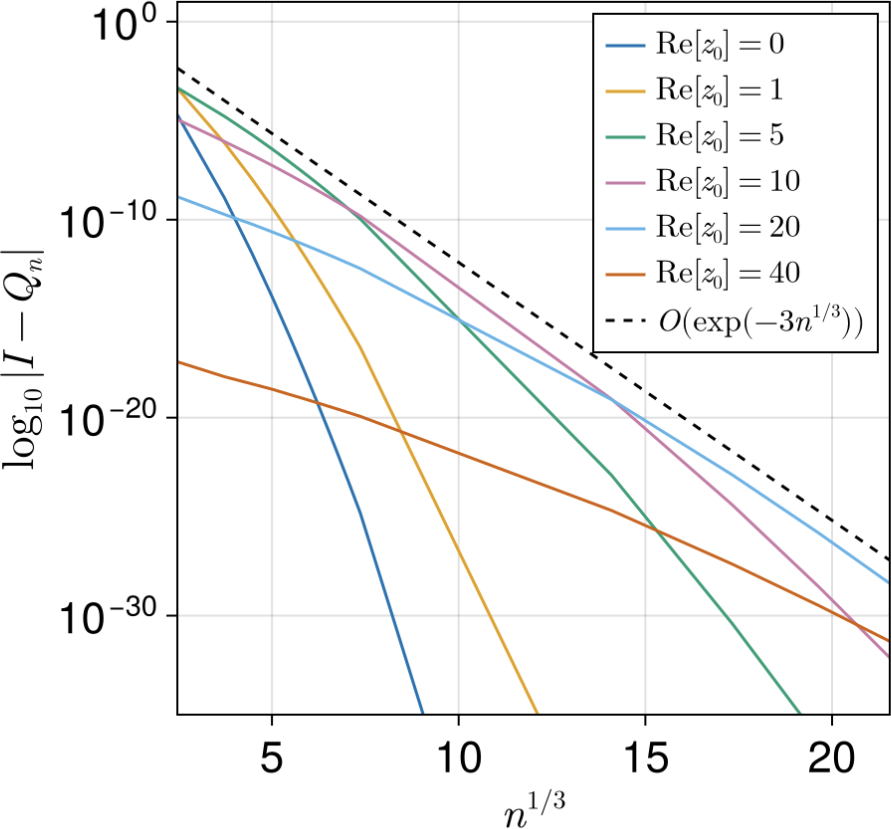}
	}
	\caption{
Plots of the Gauss-Laguerre quadrature error $|I-Q_n|$ (with $\alpha=0$) for the integrand $f(z)=(z-z_0)^{-1/2}$ for fixed $\Im[z_0]=1$ and various values of $\Re[z_0]$. 
For comparison, we also plot $C\re^{-3n^{1/3}}$ for some $C>0$.
}
	\label{f:unif}
\end{figure}

\section*{Acknowledgements}
We thank Haiyong Wang and Simon Chandler-Wilde for helpful discussions in relation to this work.

\clearpage
\appendix
\section{Appendix}
\label{s:App}

In this appendix we provide proofs of a number of auxiliary results relating to the analysis in \cite{wang_laguerre_2024}.

\subsection{Proof of \eqref{e:PhiRep}}
\label{s:PhiRep}

The proof of \eqref{e:PhiRep} given in \cite[Lemma 2.1]{wang_laguerre_2024} involved an application of L'H\^opital's rule, the details of which were not made clear. Here we present an alternative, more direct proof of \eqref{e:PhiRep}, based on an application of the Rodrigues formula \cite[(18.5.5)]{NIST:DLMF}
\begin{align}
\label{e:Rod}
L^{(\alpha)}_n(x) = \frac{1}{n!\omega_\alpha(x)}\frac{\rd^n}{\rd x^n}(\omega_\alpha(x)x^n).
\end{align}
First, we substitute \eqref{e:Rod} into \eqref{e:PhiDef} and integrate by parts $n$ times to obtain, for $z\in\C\setminus\R_+$,
\begin{align*}
\Phi^{(\alpha)}_n(z) &= \frac{\ri}{2\pi n!}\int_0^\infty \frac{1}{x-z}\frac{\rd^n}{\rd x^n}
(\omega_\alpha(x)x^n)
\, \rd x. \\
	&= \frac{\ri}{2\pi} \frac{(-1)^n}{n!}\int_0^\infty 
\omega_\alpha(x)x^n
\frac{\rd^n }{\rd x^n}\frac{1}{(x-z)}\, \rd x, \\
&= \frac{\ri}{2\pi} \int_0^\infty 
\frac{
\omega_\alpha(x)x^n	
}{(x-z)^{n+1}}\,\rd x,
\end{align*}
where we note that there are no boundary terms because the function $\omega_\alpha(x)x^n=\re^{-x}x^{n+\alpha}$ and its derivatives up to order $n-1$ vanish at both $x=0$ and $x=\infty$ (recall that $\alpha>-1$).  
Next, we temporarily restrict attention to $z\in(-\infty,0)$ and make the change of variable $x=(-z)t$ to obtain
\begin{align}
\label{e:PhiRepTemp}
\Phi^{(\alpha)}_n(z) 
&= \frac{\ri (-z)^{\alpha}}{2\pi} \int_0^\infty 
\frac{
t^{\alpha+n} \re^{zt}
}{(1+t)^{n+1}}\,\rd t,
\end{align}
for $z\in(-\infty,0)$. Then we combine \eqref{e:PhiRepTemp} with the integral representation \cite[Equation (13.4.4)]{NIST:DLMF} and the connection formula \cite[Equation (13.2.40)]{NIST:DLMF} to obtain \eqref{e:PhiRep} in the case where $z\in(-\infty,0)$. Finally, we apply analytic continuation to extend the validity of \eqref{e:PhiRep} to $z\in \C\setminus [0,\infty)$.

\subsection{Extension of \cite[Lemma 4.1]{wang_laguerre_2024}}
\label{s:Lem4p1}

In the proofs of \cite[Theorems 4.2, 4.3 \& 6.3]{wang_laguerre_2024}, the validity of Step 1 in proof Strategy \ref{st:Strat} is justified by an application of Lemma \ref{l:intlag} below. This result was stated and proved in \cite[Lemma 4.1]{wang_laguerre_2024}, but only for $\beta < 1/2$. Here we extend the result to general $\beta\in \R$. 

\begin{lem}[{Extension of \cite[Lemma 4.1]{wang_laguerre_2024}}]
\label{l:intlag}
Let $\rho>0$ and suppose that $f$ is analytic on $\overline{D_\rho}$ and satisfies $|f(z)|=O(|z|^\beta)$ as $z\to\infty$ inside $D_\rho$ for some $\beta\in \R$. Given $n\in \N$ satisfying $n> \beta-1$, and a set of $n+1$ distinct points $0<x_0<x_1<\dots<x_n<\infty$, let $p_n$ denote the unique polynomial of degree $n$ which interpolates $f$ at these points. Then for $x\in \R_+$ we have
	\begin{equation}
	\label{e:IntErr}
	 f(x)-p_n(x) = \frac{1}{2\pi \ri}\int_{P_\rho} \frac{\phi(x)f(z)}{\phi(z)(z-x)}\dd z,
	 \end{equation} 
	 where $\phi(z) = c\prod_{j=0}^n(z-x_j)$ and $c>0$ is an arbitrary constant.
\end{lem}
\begin{proof} 
Without loss of generality we can assume that $\beta\geq 0$, since for $\beta<0$ the result follows from the result for $\beta=0$. 
Fix $x\in \R_+$, and assume for the remainder of the proof that 
\begin{align}
\label{e:etaAss}
\eta \geq \max\{2x,2x_n,1,\rho^2\}.
\end{align}
Consider the two curves $\Gamma:=\{z\in P_\rho: \Re[z]\leq \eta\}$ and $\mathcal{V}:=\{z\in D_\rho: \Re[z]= \eta\}$. Their union is the boundary of a bounded domain, which by \eqref{e:etaAss} contains $x$ and all the interpolation points. 
Applying Hermite's formula for the remainder of polynomial interpolation~\cite[Theorem 3.6.1]{davis1975interpolation} gives
\begin{equation}
f(x)-p_n(x) = \frac{\phi(x)}{2\pi \ri}\int_{\Gamma \cup \mathcal{V}} \frac{f(z)}{\phi(z)(z-x)}\dd z.
\end{equation}
Hence, to prove \eqref{e:IntErr} we have to show that 
\begin{align}
\label{e:IntErr1}
\int_{P_\rho\setminus\Gamma} \frac{f(z)}{\phi(z)(z-x)}\dd z \to 0
\qquad \text{and} \qquad
\int_{\mathcal{V}} \frac{f(z)}{\phi(z)(z-x)}\dd z \to 0, \qquad \text{as} \quad\eta\to\infty.
\end{align}

For $z\in \cV\cup (\partial P_\rho \setminus \Gamma)$ and $x'\in\{x,x_0,x_1,\ldots,x_n\}$ we have by \eqref{e:etaAss} that
\begin{align}
\label{e:eta1}
|z-x'|
\geq \eta/2,
\end{align}
and also that 
\begin{align}
\label{e:eta2}
|z-x'|\geq \Re[z]/2 \geq (\Re[z]+\eta)/4\geq (\Re[z]+\rho^2)/4.
\end{align}
Furthermore, the assumption on $f$ implies that there exists $C>0$ such that $|f(z)|\leq C|z|^\beta$ for $|z|\geq 1$. Hence by \eqref{e:etaAss} we have that 
\begin{align}
\label{e:f1}
|f(z)|\leq C|z|^\beta
\end{align}
for $z\in \cV\cup (\partial P_\rho \setminus \Gamma)$. 

To estimate the integral over $\cV$, we first note using \eqref{e:PrhoCart} that the intersection points between $\cV$ and $P_\rho$ are at $z=\eta \pm 2\rho\sqrt{\rho^2+\eta}\,\ri$. Hence $|\cV|=4\rho\sqrt{\rho^2+\eta}$, which, by \eqref{e:etaAss}, gives $|\cV|\leq 4\sqrt{2} \rho \sqrt{\eta}$.
Furthermore, for $z\in\cV$ we have that $|z|\leq \sqrt{\eta^2 + 4\rho^2(\rho^2+\eta)} = \eta+2\rho^2$, 
which by \eqref{e:etaAss} implies $|z|\leq 3\eta$. Combining these observations with \eqref{e:eta1} and \eqref{e:f1} gives
\begin{equation}
\int_{\mathcal{V}} \left|\frac{f(z)}{\phi(z)(z-x)}\right|\,\rd s \leq 
\frac{2^{n+9/2}3^\beta C\rho}{c}\eta^{\beta-3/2-n},
\end{equation}
which implies the first part of \eqref{e:IntErr1}, since $n>\beta-1>\beta-3/2$. 

To estimate the integral over $\partial P_\rho \setminus\Gamma$ we use the parametrisation \eqref{e:WangParam} for $|t|\geq \sqrt{1+\eta/\rho^2}$. For such $t$ we have that 
$|z(t)| = t^2+\rho^2 \leq (1+\rho^4/(\rho^2+\eta))t^2$,
which by \eqref{e:etaAss} gives $|z(t)| \leq (1+\rho^2/2)t^2$. Similarly, for such $t$ we have that $|z'(t)|=2\sqrt{t^2+\rho^2}\leq 2\sqrt{1+\rho^2/2}\,t$. Combining these observations with \eqref{e:eta2} and \eqref{e:f1} gives
\begin{equation}
\label{e:Int1}
\int_{\partial P_\rho \setminus\Gamma} \left|\frac{f(z)}{\phi(z)(z-x)}\right|\,\rd s \leq 
2^{2n+3}\frac{C}{c}(1+\rho^2/2)^{\beta+1/2} \int_{\sqrt{1+\eta/\rho^2}}^\infty t^{2\beta-2n-3}\,\rd t.
\end{equation}
The integral on the right-hand side of \eqref{e:Int1} converges because $n>\beta-1$, and tends to zero as $\eta\to\infty$ because the lower limit of integration tends to infinity as $\eta\to\infty$, implying the second part of  \eqref{e:IntErr1}. 

\end{proof}

\subsection{Estimating the constant $\cK$}
\label{s:ConstEst}

The results in Results \ref{r:WangRes} aim to prove bounds of the form \eqref{e:FBound1}, involving a constant $\cK$ of the form
\begin{align}
\label{e:KDef2}
\mathcal{K}=\int_{P_\rho} |\re^{-az}||z|^{b}|f(z)|\,\rd s
\end{align}
for some $a,b\in\R$ such that \eqref{e:KDef2} converges.  
Under the assumption that $|f(z)|\leq C|z|^\beta$ for $z\in P_\rho$ one can use the following lemma, with $a=1$ and $d=b+\beta$, to derive upper bounds for the value of $\mathcal{K}$ \dhc{that are explicit in their $\rho$ dependence.} %

\dhc{
\begin{lem}
\label{l:KConst}
Let $\rho>0$, $a>0$ and $d\in \R$. Then there exists $C>0$, depending only on $a$ and $d$, such that
\begin{align}
\label{e:K1}
\int_{P_\rho} |\re^{-az}||z|^d\,\rd s \leq 
C\re^{a\rho^2} 
\begin{cases}
\rho^{2d+1}, & d<-1/2,\\
1+\rho^{2d+1}, & d\geq -1/2,
		\end{cases}.
\end{align}
\end{lem}
\begin{proof}
Using the parametrization \eqref{e:WangParam} we find that
\begin{align}
\label{e:K2}
\int_{P_\rho} |\re^{-az}||z|^d\,\rd s = 4\re^{a\rho^2} \int_0^\infty \re^{-at^2}(t^2+\rho^2)^{d+1/2}\,\rd t.
\end{align}
If $d+1/2 < 0$ we bound $(t^2+\rho^2)^{d+1/2}\leq \rho^{2d+1}$ and obtain
\begin{align}
\int_{P_\rho} |\re^{-az}||z|^d\,\rd s \leq 4\re^{a\rho^2} \rho^{2d+1}\int_0^\infty \re^{-at^2}\,\rd t,
\end{align}
while if $d+1/2\geq 0$ then we bound 
$(t^2+\rho^2)^{d+1/2}\leq 2^{d+1/2}(t^{2d+1} + \rho^{2d+1})$ and obtain
\begin{align}
\int_{P_\rho} |\re^{-az}||z|^d\,\rd s \leq 2^{2+d+1/2}\re^{a\rho^2} \left(\int_0^\infty \re^{-at^2}\,\rd t + \rho^{2d+1} \int_0^\infty \re^{-at^2}t^{2d+1}\,\rd t\right).
\end{align}
\end{proof}
}

\dhc{
\begin{lem} \label{l:brb}
Given $z_0\in \C\setminus \R_+$ with $\Im[z_0]\geq 0$, let $\rho:=\Re[\sqrt{-z_0}]$ and $\gamma:=\{z\in P_{\rho}: \Im[z]\geq \Im[z_0]\}$ 
Then,
 for all $\alpha\in \R$ and $\beta>-1$ there exists a constant $C>0$, depending only on $\alpha$ and $\beta$, such that, with $\mu:=\beta/2+\alpha+1/2$,
	\begin{equation}
		\int_\gamma |\re^{-z}||z^\alpha| |z-z_0|^\beta \,\rd s \leq C \re^{-\Re[z_0]}\begin{cases}
|z_0|^{\mu}, & \mu<0,\\
1+|z_0|^{\mu}, & \mu\geq 0,
		\end{cases}.
	\end{equation}
\end{lem}
\begin{proof}
\dhc{We parametrize $\gamma$ by $t\mapsto z(-t)$ for $t\in[t_0,\infty)$, where $z(\cdot)$ is as defined in~\eqref{e:WangParam} and $t_0\geq 0$ is such that $z(-t_0)=z_0$. Using the fact that 
$|z-z_0|=|z(-t)-z(-t_0)| = |t-t_0| \sqrt{(t+t_0)^2+4\rho^2}$ 
for any 
$z = z(-t)$ on $\gamma$, 
this gives
\begin{align}
I_\gamma:=	\int_\gamma |e^{-z}||z^\alpha| |z-z_0|^\beta \dd s 
&= 2 e^{\rho^2}\int_{t_0}^\infty e^{-t^2}|t-t_0|^\beta ((t+t_0)^2+4\rho^2)^{\beta/2}(t^2+\rho^2)^{\alpha+1/2} \rd t.
\end{align}
Noting that $t^2+\rho^2\leq (t+t_0)^2+4\rho^2 \leq 4(t^2+\rho^2)$ for $t\geq t_0$, 
\begin{align}
I_\gamma
&\leq 2^{1+\max(\beta,0)} \re^{\rho^2}\int_{t_0}^\infty \re^{-t^2}|t-t_0|^\beta (t^2+\rho^2)^{\mu}\, \rd t\\
& \leq 2^{1+\max(\beta,0)} \re^{\rho^2-t_0^2}\int_{0}^\infty \re^{-s^2}s^\beta ((t_0+s)^2+\rho^2)^{\mu}\, \rd s,
\end{align}
where 
$\mu:=\beta/2+\alpha+1/2$. If $\mu<0$ we find 
\begin{align}
I_\gamma& \leq 2^{1+\max(\beta,0)} \re^{\rho^2-t_0^2}(t_0^2+\rho^2)^{\mu} \int_{0}^\infty \re^{-s^2}s^\beta \,\rd s,
\end{align}
while if $\mu\geq 0$ then, noting that $(t_0+s)^2+\rho^2\leq 2(t_0^2+s^2+\rho^2)$, 
\begin{align}
I_\gamma& \leq 2^{1+\max(\beta,0)+2\mu} \re^{\rho^2-t_0^2}\left((t_0^2+\rho^2)^\mu \int_{0}^\infty \re^{-s^2}s^\beta \, \rd s + \int_{0}^\infty \re^{-s^2}s^{\beta + 2\mu}\, \rd s\right).
\end{align}
The result then follows by noting that $t_0^2-\rho^2=\Re[z_0]$ and $t_0^2+\rho^2=|z_0|$.

}

\end{proof}
}

\bibliography{references}
\bibliographystyle{siam}

\end{document}